\documentclass[12pt,a4paper]{article}
\usepackage{amsmath,amssymb,amsthm,bm}
\usepackage{graphicx}
\usepackage{txfonts}

\newtheorem{theorem}{Theorem}[section]
\newtheorem{proposition}[theorem]{Proposition}

\newtheorem{lemma}[theorem]{Lemma}
\newtheorem{remark}[theorem]{Remark}

\newtheorem{definition}[theorem]{Definition}

\newcommand{\1}{{\bm 1}}

\numberwithin{equation}{section}

\begin{document}

\begin{center}
\large\bf
Primary Non-QE Graphs on Six Vertices
\end{center}

\bigskip

\begin{center}
Nobuaki Obata\\
Graduate School of Information Sciences \\
Tohoku University\\
Sendai 980-8579 Japan \\
obata@tohoku.ac.jp
\end{center}

\bigskip

\begin{quote}
\textbf{Abstract}\enspace
A connected graph is called of non-QE class if
it does not admit a quadratic embedding in a Euclidean space.
A non-QE graph is called primary if it does not contain
a non-QE graph as an isometrically embedded proper subgraph.
The graphs on six vertices are completely classified 
into the classes of QE graphs,
of non-QE graphs, and of primary non-QE graphs.
\end{quote}

\begin{quote}
\end{quote}

\begin{quote}
\end{quote}

%%%%%%%%%%%%%%%%%%%%%%
\section{Introduction}

A \textit{quadratic embedding} of a (finite connected) graph $G=(V,E)$
in a Euclidean space $\mathbb{R}^N$ is a map
$\varphi:V\rightarrow \mathbb{R}^N$ satisfying
\[
\|\varphi(x)-\varphi(y)\|^2
=d(x,y),
\qquad
x,y\in V,
\]
where the left-hand side is the square of the Euclidean distance
between two points $\varphi(x)$ and $\varphi(y)$,
and the right-hand side is the graph distance.
We say that a graph $G$ is \textit{of QE class} or \textit{of non-QE class}
according as it admits a quadratic embedding or not.
The concept of quadratic embedding (of a metric space in general)
traces back to the early works of 
Schoenberg \cite{Schoenberg1935,Schoenberg1938}
and has been discussed along with the so-called Euclidean distance geometry
\cite{Alfakih2018, Balaji-Bapat2007, Jaklic-Modic2013, Jaklic-Modic2014, 
Liberti-Lavor-Maculan-Mucherino2014},
for wider aspects see also \cite{Deza-Laurent1997,
Graham-Winkler1985, Maehata2013} and references cited therein.

There are interesting questions both on graphs of QE class 
and on those of non-QE class.
One of the important questions on non-QE graphs 
is to obtain a sufficiently rich list of non-QE graphs.
If a graph $H$ is isometrically embedded in a graph $G$
and if $G$ is of QE class, so is $H$.
In other words, if $H$ is isometrically embedded in a graph $G$
and if $H$ is of non-QE class, so is $G$.
Thus, upon classifying graphs of non-QE class
it is important to specify a \textit{primary} non-QE graph,
that is, a non-QE graph $G$ which does not 
contain a non-QE graph $H$ as an
isometrically embedded proper subgraph.

In this paper we complete the classification of graphs on six vertices,
which is seen as a milestone of our new approach of 
the QE constant explained below.
As a result, among 112 graphs on six vertices
there are 3 primary non-QE graphs,
24 non-primary non-QE graphs,
and 85 QE graphs.
Employing the list of small connected graphs due to McKay \cite{McKay},
which is reproduced in Appendix C,
the results are stated as follows.

\begin{theorem}\label{01thm:primary non-QE graphs}
Among 112 graphs on six vertices 
the primary non-QE graphs are G6-30, G6-60 and G6-84,
see Figures \ref{fig:Non-QE graphs on 6 vertices (1)}
and \ref{fig:Non-QE graphs on 6 vertices (2)}.
\end{theorem}

\begin{theorem}\label{01thm:non-primary non-QE graphs}
Among 112 graphs on six vertices 
the non-primary non-QE graphs are
G6-22, G6-36, G6-40, G6-45, G6-49, 
G6-53, G6-54, G6-55, G6-64, G6-67, G6-71, G6-72, G6-73,
G6-78, G6-85, G6-86, G6-88, G6-91, G6-94, G6-96,
G6-101, G6-102, G6-103 and G6-107.
Moreover, these graphs contain G5-10 or G5-17 as an
isometrically embedded proper subgraph, 
see Figure \ref{fig:Non-QE graphs on 5 vertices}.
\end{theorem}

\begin{theorem}\label{01thm:QE graphs}
Among 112 graphs on six vertices 
the graphs not listed in
Theorems \ref{01thm:primary non-QE graphs}
and \ref{01thm:non-primary non-QE graphs}
are of QE class.
\end{theorem}

\begin{figure}[hbt]
\begin{center}
\includegraphics[width=60pt]{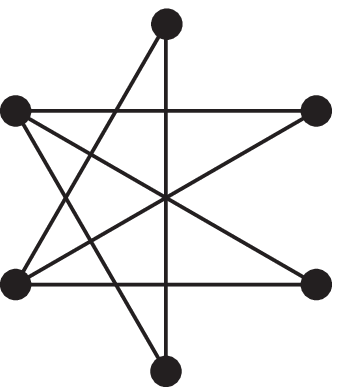}
\quad \raisebox{20pt}{$\cong$} \quad
\includegraphics[width=64pt]{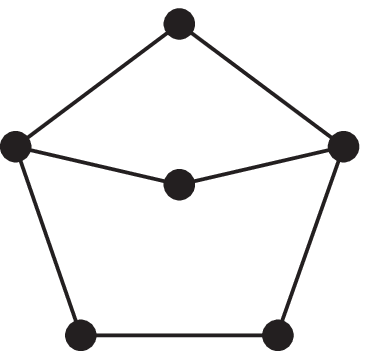}
\qquad
\includegraphics[width=60pt]{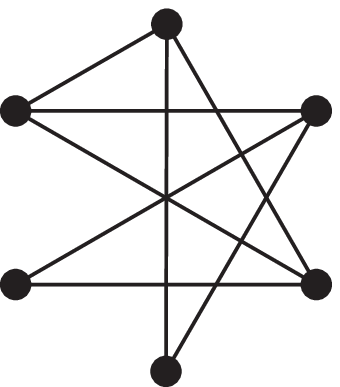}
\quad \raisebox{20pt}{$\cong$} \quad
\includegraphics[width=60pt]{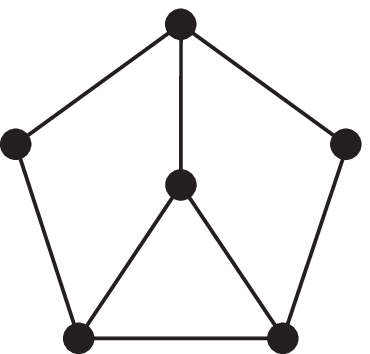}
\end{center}
\caption{G6-30 (left) and G6-60 (right)}
\label{fig:Non-QE graphs on 6 vertices (1)}
\end{figure}

\begin{figure}[hbt]
\begin{center}
\includegraphics[width=60pt]{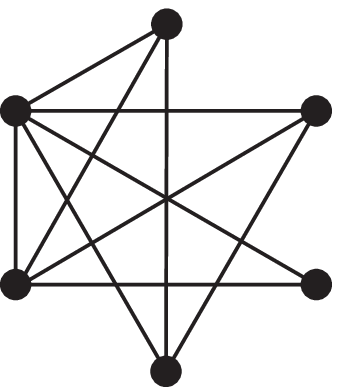}
\quad \raisebox{20pt}{$\cong$} \quad
\includegraphics[width=64pt]{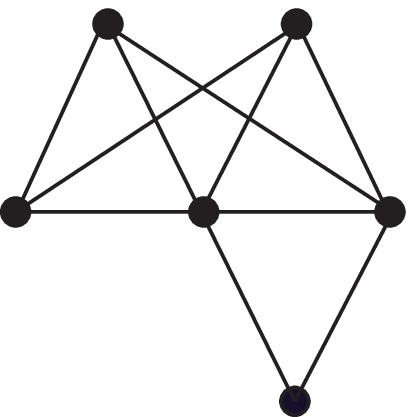}
\end{center}
\caption{G6-84}
\label{fig:Non-QE graphs on 6 vertices (2)}
\end{figure}

We briefly mention our quantitative approach.
Let $G=(V,E)$ be a graph (always assumed to be finite and connected).
Let $C(V)$ be the space of $\mathbb{R}$-valued functions on $V$
equipped with
the canonical inner product denoted by $\langle\cdot,\cdot\rangle$.
The distance matrix $D=[d(x,y)]_{x,y\in V}$ acts on 
$C(V)$ from the left as usual.
It follow from Schoenberg's theorem 
\cite{Schoenberg1935,Schoenberg1938}
that $G$ is of QE class
if and only if the distance matrix  $D=[d(x,y)]_{x,y\in V}$ is
conditionally negative definite, namely,
$\langle f, Df\rangle\le0$ for all $f\in C(V)$ with
$\langle \mathbf{1},f\rangle=0$,
where $\mathbf{1}\in C(V)$ stands for 
the constant function taking value one.
It is then natural to consider the conditional maximum
defined by
\begin{equation}\label{1eqn:def od QEC(G)}
\mathrm{QEC}(G)
=\max\{\langle f,Df \rangle\,;\, 
\langle f,f \rangle=1, \, \langle \1,f \rangle=0\},
\end{equation}
which is called 
the \textit{quadratic embedding constant} 
(\textit{QE constant} for short) of $G$.
Apparently, the QE constant provides a criterion for a graph
to be of QE class or not.
In fact, a graph $G$ on at least two vertices is of QE class
if and only if $\mathrm{QEC}(G)\le0$.

In the recent paper \cite{Obata-Zakiyyah2018} 
we started a systematic study of the QE constant
as a new quantitative invariant of a graph,
where we obtained basic formulas as well as examples.
Since then we have collected explicit values of the QE constants
of particular graphs \cite{Irawan-Sugeng2021,
Mlotkowski2022,Obata2017,Purwaningsih-Sugeng2021}
and formulas in relation to graph operations
\cite{Lou-Obata-Huang2022, MO-2018}.
We are also interested in classifying
the finite connected graphs in terms of the QE constants
\cite{Baskoro-Obata2021}, where
we started an attempt to characterize graphs along with
the increasing sequence of the QE constants of paths.

Moreover, the QE constant is interesting from the point of
view of spectral analysis of distance matrices
\cite{Aouchiche-Hansen2014, Xue-Liu-Shu2020}.
In fact, $\mathrm{QEC}(G)$ lies between
the largest and the second largest eigenvalues of 
the distance matrix $D$
in such a way that $\lambda_2(D)\le \mathrm{QEC}(G)<\lambda_1(D)$.
It is an interesting open question to characterize
graphs $G$ with $\lambda_2(D)=\mathrm{QEC}(G)$.
Any transmission-regular graph,
of which the distance matrix has a constant row sum by definition,
has this property and also any path $P_{2n}$ 
on even vertices \cite{Mlotkowski2022}.
While, the equality holds for any distance regular graph
and all distance-regular graphs with $\mathrm{QEC}(G)\le0$
are listed in \cite{Koolen-Shpectorov1994}.

As another interesting aspect of the QE constant
we mention a relation to noncommutative harmonic analysis.
The entry-wise exponential matrix of the distance matrix
is called the \textit{$Q$-matrix}
and is defined by
$Q=Q_q=[q^{d(x,y)}]_{x,y\in V}$, where $-1\le q\le 1$.
Let $\pi(G)\subset [-1,1]$ be the region of $q$ such
that $Q=Q_q$ is positive definite (allowing zero eigenvalue).
By the well-known general theory on
positive definite kernels
we know that $D$ is conditionally negative definite
if and only if $Q=Q_q$ is positive definite for all $0\le q\le1$,
that is $[0,1]\subset \pi(G)$.
The last condition is essential for noncommutative
harmonic analysis, first appeared in \cite{Bozejko89}
in relation to the length function of a free group,
and also for $q$-deformed spectral analysis of growing graphs
\cite{Hora-Obata2007}.
Here an interesting question
related to $\mathrm{QEC}(G)$ arises
to determine the positivity region $\pi(G)$.
Some concrete examples are found in \cite{Obata2011}
and the graphs $G$ with $\pi(G)=[-1,1]$
are recently characterized in \cite{Tanaka2022}.

This paper is organized as follows.
In Section 2 we prepare basic notions and notations,
for more details see \cite{MO-2018,Obata-Zakiyyah2018}.
In Section 3 we mention some criteria
for a graph to be of QE or of non-QE class
and sieve the graphs on six vertices.
In particular, we achieve Theorem \ref{01thm:non-primary non-QE graphs}
by showing that a non-primary non-QE graph
on six vertices contains G5-10 or G5-17 (these are
the primary non-QE graphs on five vertices)
as an isometrically embedded subgraph.
In Section 4 we complete the classification
by calculating the QE constants of the remainders
which are not judged during the previous section.
Appendices A and B contain formulas for
$\mathrm{QEC}(K_n\wedge K_{m,1})$ and
$\mathrm{QEC}(K_n\backslash P_4)$, respectively.
Appendix C contains the list of graphs on six vertices,
which is reproduced from McKay \cite{McKay}.

\bigskip
{\bfseries Notation:}
We write ``Gm-n'' for the graph on $m$ vertices
with the item number $n$ in the list of
small connected graphs due to McKay \cite{McKay}.
The list of graphs on six vertices is reproduced in
Appendix C.

\bigskip
{\bfseries Acknowledgements:}
This work is supported by
JSPS Grant-in-Aid for Scientific Research No.~19H01789.

\section{Quadratic Embedding of Graphs}

\subsection{Distance Matrices}

A \textit{graph} $G=(V,E)$ is a pair of 
a non-empty set $V$ of vertices and a set $E$ of edges,
where an edge is a set of two distinct vertices in $V$.
For $x,y\in V$ we write $x\sim y$ if $\{x,y\}\in E$.
A graph is called \textit{connected} if 
any pair of vertices $x,y\in V$ there exists
a finite sequence of vertices $x_0,x_1,\dots,x_m\in V$ 
such that $x=x_0\sim x_1\sim \dotsb\sim x_m=y$.
In that case the sequence of vertices is called
a \textit{walk} from $x$ to $y$ of length $m$.
Unless otherwise stated, by a graph we mean a finite connected graph
throughout this paper.

Let $G=(V,E)$ be a graph (assumed to be finite and connected).
For $x,y \in V$ with $x\neq y$ let $d(x,y)=d_G(x,y)$ denote
the length of a shortest walk connecting $x$ and $y$.
By definition we set $d(x,x)=0$.
Then $d(x,y)$ becomes a metric on $V$,
which we call the \textit{graph distance}.
The \textit{distance matrix} of $G$ is defined by
\[
D=[d(x,y)]_{x,y\in V}\,,
\]
which is a matrix with index set $V\times V$.

Let $C(V)$ be the linear space 
of $\mathbb{R}$-valued functions $f$ on $V$.
We always identify $f\in C(V)$ with a column vector 
$[f_x]_{x\in V}$ through $f_x=f(x)$. 
The canonical inner product on $C(V)$ is defined
\[
\langle f,g \rangle
=\sum_{x\in V} f(x)g(x)\,,
\qquad f,g\in C(V).
\]
As usual, the distance matrix $D$  gives rise to 
a linear transformation on $C(V)$
by matrix multiplication.
Since $D$ is symmetric, we have
$\langle f,Dg \rangle=\langle Df,g \rangle$.

\subsection{Quadratic Embedding}

A \textit{quadratic embedding} of a graph $G=(V,E)$ 
in a Euclidean space $\mathbb{R}^N$ is
a map $\varphi:V\rightarrow \mathbb{R}^N$ satisfying
\[
\|\varphi(x)-\varphi(y)\|^2
=d(x,y),
\qquad
x,y\in V,
\]
where the left-hand side is the square of the Euclidean distance
between two points $\varphi(x)$ and $\varphi(y)$.
A graph $G$ is called \textit{of QE class} or \textit{of non-QE class}
according as it admits a quadratic embedding or not.

In general, a graph $H=(V^\prime, E^\prime)$ is called a
\textit{subgraph} of $G=(V,E)$
if $V^\prime\subset V$ and $E^\prime\subset E$.
If both $G$ and $H$ are connected,
they have their own graph distances.
If they coincide in such a way that
\[
d_H(x,y)=d_G(x,y),
\qquad x,y\in V^\prime,
\]
we say that $H$ is \textit{isometrically embedded} in $G$.
The next assertion is obvious by definition but useful.

\begin{lemma}\label{02lem:heredity}
Let $G=(V,E)$ and $H=(V^\prime,E^\prime)$ be graphs
and assume that $H$ is isometrically embedded in $G$.
\begin{enumerate}
\setlength{\itemsep}{0pt}
\item[\upshape (1)] If $G$ is of QE class, so is $H$.
\item[\upshape (2)] If $H$ is of non-QE class, so is $G$.
\end{enumerate}
\end{lemma}

The following definition is then
adequate for classifying graphs of non-QE class.

\begin{definition}
\normalfont
A graph of non-QE class is called \textit{primary}
if it contains no isometrically embedded proper subgraphs
of non-QE class.
\end{definition}

Recall that the \textit{diameter} of a graph $G=(V,E)$ is defined by
\[
\mathrm{diam}(G)=\max\{d(x,y)\,;\, x,y\in V\}.
\]
The next result is useful,
of which the proof is straightforward, see also \cite{Obata2022}.

\begin{lemma}\label{02lem:isometric embedding}
Let $G=(V,E)$ be a graph and $H=(V^\prime,E^\prime)$ a subgraph.
\begin{enumerate}
\setlength{\itemsep}{0pt}
\item[\upshape (1)] If $H$ is isometrically embedded in $G$,
then $H$ is an induced subgraph of $G$.
\item[\upshape (2)] If $H$ is an induced subgraph of $G$ and
$\mathrm{diam\,}(H)\le2$, then $H$ is isometrically embedded in $G$.
\end{enumerate}
\end{lemma}

%%%%%%%%%%%%%%%%%%%%%%%%%%%%%%%%%%%%%%%%%%
\subsection{Quadratic Embedding Constants}

Let $G=(V,E)$ be a graph with $|V|\ge2$.
The \textit{quadratic embedding constant}
(\textit{QE constant} for short) of $G$ is defined by
\[
\mathrm{QEC}(G)
=\max\{\langle f,Df \rangle\,;\, f\in C(V), \,
\langle f,f\rangle=1, \, \langle \1,f \rangle=0\},
\]
where $\1\in C(V)$ is defined by $\1(x)=1$ for all $x\in V$.

It follows from Schoenberg \cite{Schoenberg1935,Schoenberg1938} 
that a graph $G$ admits a quadratic embedding
if and only if the distance matrix $D=[d(x,y)]$ is
conditionally negative definite, i.e.,
$\langle f,Df \rangle\le0$ 
for all $f\in C(V)$ with $\langle \1,f \rangle=0$.
Hence a graph $G$ is of QE class (resp. of non-QE class)
if and only if $\mathrm{QEC}(G)\le0$ (resp. $\mathrm{QEC}(G)>0$).

In order to get the value of $\mathrm{QEC}(G)$ we have
established a standard method 
on the basis of Lagrange's multipliers.

\begin{proposition}[\cite{Obata-Zakiyyah2018}]
\label{02prop:QEC}
Let $D$ be the distance matrix of a graph $G$ 
on $n$ vertices with $n\ge3$.
Let $\mathcal{S}$ be the set of all stationary points  
$(f,\lambda,\mu)$ of
\begin{equation}\label{02eqn:basic Lagrange}
\varphi(f,\lambda,\mu)
=\langle f,Df\rangle
 -\lambda(\langle f,f\rangle-1)
 -\mu\langle \1,f\rangle,
\end{equation}
where $f\in C(V)\cong \mathbb{R}^n$, $\lambda\in\mathbb{R}$
and $\mu\in\mathbb{R}$.
Then we have
\[
\mathrm{QEC}(G)
=\max\{\lambda\,;\, (f,\lambda,\mu)\in \mathcal{S}\}.
\]
\end{proposition}

For the range of $\mathrm{QEC}(G)$ we have the following results.

\begin{proposition}[\cite{Baskoro-Obata2021}]
\label{02prop:QEC ge -1}
For any graph $G$ we have $\mathrm{QEC}(G)\ge -1$ and
the equality holds if and only if $G$ is a complete graph.
\end{proposition}

Thus, in order to determine $\mathrm{QEC}(G)$ of 
a graph $G$ which is not a complete graph,
it is sufficient to seek out the stationary points of
$\varphi(f,\lambda,\mu)$ with $\lambda>-1$ and 
then to specify the maximum of $\lambda$ appearing therein.

A table of the QE constants of graphs on $n\le5$ vertices 
is available \cite{Obata-Zakiyyah2018}.
By direct observation we have the following

\begin{theorem}[\cite{Obata-Zakiyyah2018}]
\label{02thm:G5-10 and 17}
All graphs on $n\le4$ vertices are of QE class.
There are just two non-QE graphs on five vertices, 
that are G5-10 and G5-17, 
see Figure \ref{fig:Non-QE graphs on 5 vertices}.
Their QE constants are given by
\[
\mathrm{QEC}(\text{\upshape G5-10})=\frac25,
\qquad
\mathrm{QEC}(\text{\upshape G5-17})=\frac{4}{11+\sqrt{161}}
\approx 0.1689.
\]
Moreover, both are primary non-QE graphs.
\end{theorem}

\begin{figure}[hbt]
\begin{center}
\includegraphics[width=60pt]{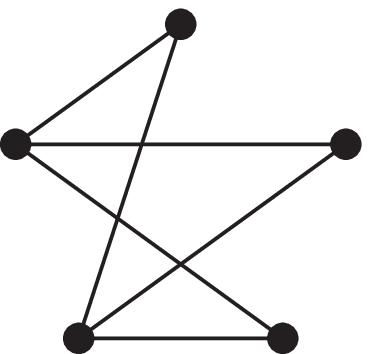}
\quad \raisebox{20pt}{$\cong$} \quad
\includegraphics[width=60pt]{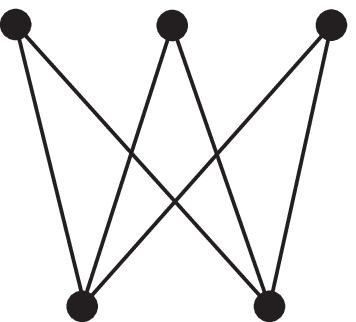}
\qquad
\includegraphics[width=60pt]{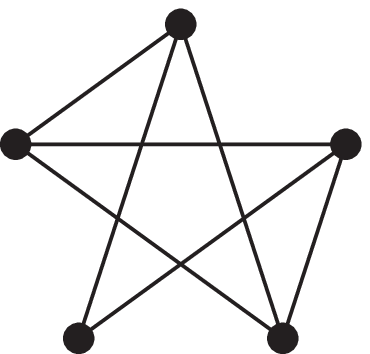}
\quad \raisebox{20pt}{$\cong$} \quad
\includegraphics[width=60pt]{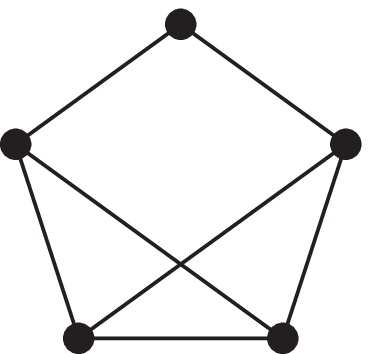}
\end{center}
\caption{G5-10 (left) and G5-17 (right)}
\label{fig:Non-QE graphs on 5 vertices}
\end{figure}

\begin{lemma}\label{02lem:isometrically embedded subgraphs}
Let $G=(V,E)$ and $H=(V^\prime,E^\prime)$ be two graphs
with $|V|\ge2$ and $|V^\prime|\ge2$.
If $H$ is isometrically embedded in $G$, we have 
$\mathrm{QEC}(H)\le \mathrm{QEC}(G)$.
\end{lemma}

The proof is obvious as the distance matrix of $H$
is a principal submatrix of the one of $G$,
see \cite{Obata-Zakiyyah2018}.
Note that Lemma \ref{02lem:heredity} follows 
by comparison of the QE constants.

%%%%%%%%%%%%%%%%%%%%%%%%%%%%%%%%%%%%%%%%%%%%%%%%%%%%%%%%%%%
\section{Connected Graphs on Six Vertices}

We will classify the graphs on six vertices
into the classes of QE graphs,
of non-QE graphs, and of primary non-QE graphs
along the following steps:
\begin{description}
\setlength{\itemsep}{0pt}
\item[Step 1] Sieve out all graphs which are the star products 
or Cartesian products of two graphs of QE class. 
Those graphs are of QE class.
\item[Step 2] Sieve out all non-primary non-QE graphs.
\item[Step 3] Use some special series of graphs of which
the QE constants are known.
\item[Step 4] Sieve out all graphs which are graph joins
of two regular graphs.
\item[Step 5] Construct explicitly quadratic embeddings 
from known ones of smaller graphs.
\item[Step 6] Remaining graphs are judged by
explicit calculation of the QE constant.
\end{description}

\subsection{Star and Cartesian Products}

Let $G_1=(V_1,E_1)$ and $G_2=(V_2,E_2)$ be two graphs.
The \textit{Cartesian product} of $G_1$ and $G_2$ is a graph on
the vertex set $V=V_1\times V_2$ with the adjacency relation
\[
(x_1,y_1)\sim (x_2,y_2)
\quad\Longleftrightarrow\quad
x_1=x_2, \,\, y_1\sim y_2 \quad\text{or}\quad
x_1\sim x_2,\,\, y_1=y_2.
\]
The Cartesian product of $G_1$ and $G_2$ 
is denoted by $G_1\times G_2$.
Assume that $V_1$ and $V_2$ are disjoint,
and choose $o_1\in V_1$ and $o_2\in V_2$ as
distinguished vertices (or roots) of $G_1$ and $G_2$,
respectively.
Then the \textit{star product} of $G_1$ and $G_2$
with respect to $o_1$ and $o_2$ is a graph
obtained by putting two graphs together at 
the distinguished vertices $o_1$ and $o_2$,
which are then identified.
The star product depends on the choice of distinguished vertices
and is denoted by $G_1\star_{(o_1,o_2)} G_2$.
For simplicity we write $G_1\star G_2$ 
whenever there is no danger of confusion.

\begin{proposition}[\cite{Obata-Zakiyyah2018}]
\label{03prop:Cartesian products of QE}
If two graphs $G_1$ and $G_2$ are of QE class,
so is their Cartesian product $G_1\times G_2$.
\end{proposition}

\begin{proposition}[\cite{Obata-Zakiyyah2018}]
\label{03prop:star products of QE}
If two graphs $G_1$ and $G_2$ are of QE class,
so is their star product $G_1\star G_2$.
\end{proposition}

\begin{remark}
\normalfont
By definition any graph $G$ admits 
a Cartesian product structure
as $G=G\times K_1$, where $K_1$ is the graph consisting of
a single vertex,
and similarly a star product structure
$G=G\star K_1$.
Those trivial cases are not 
excluded from Propositions \ref{03prop:Cartesian products of QE}
and \ref{03prop:star products of QE}.
In fact, a quadratic embedding of $K_1$ in a Euclidean space
is trivial.
However, for the QE constant we need to avoid $K_1$ since
$\mathrm{QEC}(K_1)$ is not defined.
\end{remark}

\begin{remark}
\normalfont
It is shown \cite{Obata-Zakiyyah2018} that
$\mathrm{QEC}(G_1\times G_2)=0$
for any non-trivial Cartesian product
of two QE graphs $G_1$ and $G_2$.
On the other hand, it is a highly non-trivial problem
to describe $\mathrm{QEC}(G_1\star G_2)$ in terms of
$\mathrm{QEC}(G_1)$ and $\mathrm{QEC}(G_2)$.
Some useful estimates are known, 
see \cite{Baskoro-Obata2021, MO-2018}.
\end{remark}

In view of the unique non-trivial factorization $6=2\times 3$
we can easily specify all graphs on six vertices 
which admit non-trivial Cartesian product structures.
It is also a simple task to
check whether a graph on six vertices admits a
non-trivial star product structure or not.
If $G=G_1\times G_2$ or $G=G_1\star G_2$ is such a 
non-trivial product,
the numbers of vertices of $G_1$ and $G_2$ are less than
or equal to five.
On the other hand, we know that
any graph on less than or equal to five vertices
are of QE class except two graphs G5-10 and G5-17.
We thus collect all graphs on six vertices
which are non-trivial Cartesian or star products of two QE graphs.

\begin{remark}
\normalfont
If $G=G_1\star G_2$ and $G_1$ is one of the two graphs G5-10 and G5-17,
then $G_2$ is necessarily $K_2=P_2$
and $G_1$ is isometrically embedded in $G$.
It then follows from Lemma \ref{02lem:heredity} that
$G$ is of non-QE class.
This case will be discussed in a more general context 
in Subsection \ref{subsec:Non-Primary Non-QE Graphs}.
\end{remark}

The results are summarized in the following table.
Among 112 graphs on six vertices
there are 2 graphs which are non-trivial Cartesian products of 
two QE graphs,
and 51 graphs which are non-trivial star products of 
two QE graphs.
These two classes are mutually exclusive.
Note that the six graphs G6-1$\sim$G6-6 are trees.

\begin{center}
\renewcommand{\arraystretch}{1.2}
\begin{tabular}{|l|l|l|}
\hline
graphs & No. & QE/Non-QE \\
\hline
Cartesian products 
& G6-35 ($K_2\times P_3$), G6-80 ($K_2\times K_3$) & QE 
\\ \hline
star products & G6-1$\sim$18, G6-20, G6-21, G6-23$\sim$29, & QE
\\
of QE graphs & G6-31, G6-33, G6-34, G6-37$\sim$39,  & 
\\
& G6-41$\sim$44, G6-47, G6-48, G6-51, & 
\\
& G6-56, G6-57, G6-62, G6-68$\sim$70, &
\\
& G6-74, G6-77, G6-83, G6-89, G6-98 &  
\\ \hline
\end{tabular}
\end{center}

\subsection{Non-Primary Non-QE Graphs}
\label{subsec:Non-Primary Non-QE Graphs}

Let $G$ be a non-QE graph on six vertices.
If $G$ is not primary, it contains a non-QE graph $H$ as 
an isometrically embedded proper subgraph.
Since $H$ has five or less vertices, it is necessarily
G5-10 or G5-17.
Thus, any non-primary non-QE graph on six vertices
is obtained by adding a vertex $a$ to G5-10 or G5-17.
Since $\mathrm{diam}(\text{G5-10})
=\mathrm{diam}(\text{G5-17})=2$,
it follows from Lemma \ref{02lem:isometric embedding} that
G5-10 or G5-17 is isometrically embedded in $G$
whatever vertices are connected to $a$.
In this way the non-primary non-QE graphs on six vertices are specified.

\begin{proposition}\label{03prop:non-QE containing G5-10}
There are 11 graphs on six vertices containing G5-10 
as an isometrically embedded subgraph, 
which are classified by the degree of additional vertex $a$.

$\deg(a)=1$: G6-22, G6-36,

$\deg(a)=2$: G6-40, G6-49, G6-55, 

$\deg(a)=3$: G6-64, G6-72, G6-73,

$\deg(a)=4$: G6-85, G6-86,

$\deg(a)=5$: G6-96.
\end{proposition}

\begin{proposition}\label{03prop:non-QE containing G5-17}
There are 17 graphs on six vertices containing G5-17 
as an isometrically embedded subgraph, 
which are classified by the degree of additional vertex $a$.

$\deg(a)=1$: G6-45, G6-53, G6-54, 

$\deg(a)=2$: G6-64, G6-67, G6-71, G6-72, G6-78,  

$\deg(a)=3$: G6-85, G6-86, G6-88, G6-91, G6-94, 

$\deg(a)=4$: G6-101, G6-102, G6-103,

$\deg(a)=5$: G6-107.
\end{proposition}

The above results are summarized in the following table,
where G6-64, G6-72, G6-85 and G6-86 occur in both classes.
As a result, the non-primary non-QE graphs on six vertices 
are completely specified and the proof of
Theorem \ref{01thm:non-primary non-QE graphs} is completed.
It is also noteworthy that 
any non-QE graph on six vertices not listed here 
is a primary non-QE graph.

\begin{center}
\renewcommand{\arraystretch}{1.2}
\begin{tabular}{|l|l|l|}
\hline
graphs & No. & QE/Non-QE \\
\hline
containing G5-10 as &  G6-22, G6-36, G6-40, G6-49, G6-55, &non-QE  
\\
an isometrically & G6-64, G6-72, G6-73, G6-85, G6-86, &
\\
embedded subgraph & G6-96 &
\\ \hline
containing G5-17 as &  G6-45, G6-53, G6-54, G6-64, G6-67, &non-QE  
\\
an isometrically &  G6-71, G6-72, G6-78, G6-85, G6-86, &
\\
embedded subgraph & G6-88, G6-91, G6-94, G6-101, &
\\
& G6-102, G6-103, G6-107 &
\\ \hline
\end{tabular}
\end{center}

\subsection{Special Series of Graphs with Known QE Constants}

\begin{proposition}[\cite{Obata-Zakiyyah2018}]
For the complete graph $K_n$ with $n\ge2$ we have
\[
\mathrm{QEC}(K_n)=-1.
\]
\end{proposition}

\begin{proposition}[\cite{Mlotkowski2022}]
\label{02ex:path graphs}
\normalfont
For the paths $P_n$ with $n\ge2$ we have
\[
\mathrm{QEC}(P_n)=-\bigg(1+\cos\dfrac{\pi}{n}\bigg)^{-1}.
\]
\end{proposition}

\begin{proposition}[\cite{Obata-Zakiyyah2018}]
\label{02ex:cycle graphs}
For the cycles on odd number of vertices we have
\[
\mathrm{QEC}(C_{2n+1})=-\bigg(4\cos^2\dfrac{\pi}{2n+1}\bigg)^{-1},
\qquad n\ge1,
\]
and for those on even number of vertices we have
\[
\mathrm{QEC}(C_{2n})=0,
\qquad
n\ge2.
\]
\end{proposition}

\begin{proposition}[\cite{Obata2022}]
\label{03prop:complete multipartite graphs}
For the complete multipartite graph
$K_{m_1,m_2,\dots,m_k}$ with 
$k\ge2$ and $m_1\ge m_2\ge \dotsb \ge m_k\ge1$, we have
\[
\mathrm{QEC}(K_{m_1,m_2,\dots,m_k})=
\begin{cases}
-2+m_1, & \text{if $m_1=m_2$}, \\
-2-\alpha^*, & \text{if $m_1>m_2$},
\end{cases}
\]
where $\alpha^*$ is the minimal solution to 
\[
\sum_{i=1}^k \frac{m_i}{\alpha+m_i}=0.
\]
Note that $-m_1<\alpha^*<-m_2$.
\end{proposition}

Finally, we also use the following new result,
of which the proof and related discussion are deferred in the Appendix A.

\begin{proposition}\label{03prop:wedge product}
For $1\le m\le n$
let $K_n\wedge K_{m,1}$ be the graph obtained 
by putting $m$ vertices of the complete graph $K_n$
and the $m$ end-vertices of the star $K_{m,1}$ together.
Then we have
\[
\mathrm{QEC}(K_n\wedge K_{m,1})
=\frac{-2n+m-1+\sqrt{n(n-m)(m+1)}}{n+1}\,.
\]
\end{proposition}

It is an easy task to collect all 
complete multipartite graphs on six vertices
together with their QE constants by
Proposition \ref{03prop:complete multipartite graphs}.
Among them there are three non-QE graphs
which are not primary. 
On the other hand, 
it follows from Proposition \ref{03prop:wedge product} that
\[
\mathrm{QEC}(K_5\wedge K_{m,1})
=\frac{-11+m+\sqrt{5(5-m)(m+1)}}{6}
<0,
\qquad 1\le m\le 5.
\]
Hence $K_5\wedge K_{m,1}$ are of QE class for $1\le m\le 5$.
Note that $K_5\wedge K_{1,1}=K_5\star K_2$
and $K_5\wedge K_{5,1}=K_6$.
The results in this subsection are summarized in the following table.

\begin{center}
\renewcommand{\arraystretch}{1.2}
\begin{tabular}{|l|l|l|}
\hline
graphs & No. & QE/Non-QE \\
\hline
complete graphs & G6-112 ($K_6$) & QE 
\\ \hline
paths & G6-6 ($P_6$) & QE 
\\ \hline
cycles & G6-19 ($C_6$) & QE 
\\ \hline
complete &
G6-1 ($K_{5,1}$),
G6-61 ($K_{4,1,1}$) & QE 
\\
multipartite graphs
& G6-104 ($K_{3,1,1,1}$), G6-108 ($K_{2,2,2}$) & 
\\
& G6-110 ($K_{2,2,1,1}$), G6-111 ($K_{2,1,1,1,1}$), & 
\\ \cline{2-3}
 & G6-40 ($K_{4,2}$), G6-73 ($K_{3,3}$), & Non-QE 
\\
 & G6-96 ($K_{3,2,1}$), & 
\\ \hline
$K_5\wedge K_{m,1}$ & G6-98 ($K_5\wedge K_{1,1}$),
G6-105 ($K_5\wedge K_{2,1}$), & QE
\\
& G6-109 ($K_5\wedge K_{3,1}$), G6-111 ($K_5\wedge K_{4,1}$) &
\\
& G6-112 ($K_5\wedge K_{5,1}$) &
\\ \hline
\end{tabular}
\end{center}

\subsection{Graph Joins of Regular Graphs}

Let $G_1=(V_1,E_1)$ and $G_2=(V_2,E_2)$ be
two graphs with disjoint vertex sets.
The \textit{graph join} of $G_1$ and $G_2$ is 
a graph $G=(V,E)$, where
\[
V=V_1\cup V_2,
\qquad
E=E_1\cup E_2\cup\{\{x,y\}\,;\, x\in V_1,\,\, y\in V_2\},
\]
and is denoted by $G=G_1+G_2$.
Note that $G$ becomes a connected graph 
even if $G_1$ or $G_2$ is not connected.
Quite a few graphs admit graph join structures but
we have no general formula for $\mathrm{QEC}(G_1+G_2)$
except the case where both
$G_1$ and $G_2$ are regular graphs.

\begin{proposition}[\cite{Lou-Obata-Huang2022}]
\label{03prop:QEC(G_1+G_2)}
For $i=1,2$ let $G_i=(V_i,E_i)$ be
a $r_i$-regular graph on $n_i=|V_i|$ vertices
with $V_1\cap V_2=\emptyset$, where $r_i\ge0$ and $n_i\ge1$.
Then we have
\[
\mathrm{QEC}(G_1+G_2)=-2+
\max\left\{
-\lambda_{\min}(G_1),\,
-\lambda_{\min}(G_2),\,
\frac{2n_1n_2-r_1n_2-r_2n_1}{n_1+n_2}\right\},
\]
where $\lambda_{\min}(G_i)$ is the minimal eigenvalue of
the adjacency matrix of $G_i$.
\end{proposition}

According to the partition $6=1+5=2+4=3+3$ we can easily list
the graphs on six vertices
which are graph joins of two regular graphs.
Their QE constants are obtained easily by
the formula in Proposition \ref{03prop:QEC(G_1+G_2)}
and thereby we can classify them into the classes of QE or non-QE.
The results are summarized in the following table
without mentioning the QE constants.

\begin{center}
\renewcommand{\arraystretch}{1.2}
\begin{tabular}{|l|l|l|}
\hline
graphs & No. & QE/Non-QE \\
\hline
graph joins  & G6-112 ($K_1+K_5$), G6-92 ($K_1+C_5$)  & QE 
\\
($1+5$) & G6-1 ($K_1+\bar{K}_5$) &
\\ \hline
graph joins & G6-112 ($K_2+K_4$), G6-110 ($K_2+C_4$) & QE 
\\
($2+4$) & G6-97 ($K_2+(K_2\cup K_2)$), G6-61 ($K_2+\bar{K}_4$), &
\\
& G6-111 ($\bar{K}_2+K_4$), G108 ($\bar{K}_2+C_4$) & 
\\ \cline{2-3} 
 & G6-88 \rule[0pt]{0pt}{16pt}($\bar{K}_2+(K_2\cup K_2)$), 
 G6-40 ($\bar{K}_2+\bar{K}_4$) & Non-QE
\\ \hline
graph joins & G6-112 \rule[0pt]{0pt}{16pt} ($K_3+K_3$), G6-104 ($K_3+\bar{K}_3$) & QE 
\\ \cline{2-3}
($3+3$) & G6-73 \rule[0pt]{0pt}{16pt} ($\bar{K}_3+\bar{K}_3$) & Non-QE
\\ \hline
\end{tabular}
\end{center}

Note that the three non-QE graphs in the above table appear
already in the list of non-primary non-QE graphs,
see Subsection \ref{subsec:Non-Primary Non-QE Graphs}.

%%%%%%%%%%%%%%%%%%%
\subsection{Explicit Construction of Quadratic Embeddings}

\subsubsection{Graphs with pendant edges}

A graph $G=(V,E)$ is called 
\textit{with a pendant edge} if there exist 
four distinct vertices
$a,b,a^\prime, b^\prime\in V$ satisfying
\[
a\sim a^\prime \sim b^\prime \sim b \sim a,
\qquad
\deg(a^\prime)=\deg(b^\prime)=2.
\]
The edge $\{a^\prime, b^\prime\}$ is called a
\textit{pendant edge}.
We first note that the induced subgraph
spanned by $a,b,a^\prime$ and $b^\prime$ forms a cycle $C_4$.
Since this cycle is isometrically embedded in $G$, we have
$\mathrm{QEC}(G)\ge \mathrm{QEC}(C_4)=0$.
On the other hand, the induced subgraph spanned
by $V\backslash\{a^\prime,b^\prime\}$,
denoted by $H$, is also isometrically embedded in $G$.
Hence $\mathrm{QEC}(G)\ge \mathrm{QEC}(H)$ and we have
\begin{equation}\label{03eqn:comparison}
\max\{0,\mathrm{QEC}(H)\}
\le \mathrm{QEC}(G).
\end{equation}

\begin{proposition}[Graph with pendant edge]
\label{03prop:adding a square}
Let $G=(V,E)$ be a graph with a pendant edge 
$a^\prime\sim b^\prime$.
Let $H$ be the induced subgraph spanned 
by $V\backslash\{a^\prime, b^\prime\}$.
If $H$ is of QE class, so is $G$ and
$\mathrm{QEC}(G)=0$.
\end{proposition}

\begin{proof}
Let $d(x,y)$ denote the graph distance of $G$.
As is mentioned above, $d(x,y)$ coincides
with the graph distance of $H$.
By assumption we take a quadratic embedding of $H$, say,
$\varphi: V\backslash\{a^\prime, b^\prime\}
\rightarrow \mathbb{R}^N$.
Then we have
\[
\|\varphi(x)-\varphi(y)\|^2=d(x,y),
\qquad x,y\in V\backslash\{a^\prime, b^\prime\}.
\]
In view of the natural inclusion $\mathbb{R}^N\subset
\mathbb{R}^{N+1}$, we define 
a map $\Tilde{\varphi}:V\rightarrow \mathbb{R}^{N+1}$ by
\begin{equation}\label{03def:quadratic embedding Lemma3.3}
\Tilde{\varphi}(x)=
\begin{cases} 
(\varphi(x),0), & x\neq a^\prime,b^\prime, \\
(\varphi(a),1), & x=a^\prime, \\
(\varphi(b),1), & x=b^\prime.
\end{cases}
\end{equation}
We will prove that $\Tilde{\varphi}$ gives rise to a 
quadratic embedding of $G$
by showing that
\begin{equation}\label{03:adding square quadratic embedding}
\|\Tilde{\varphi}(x)-\Tilde{\varphi}(y)\|^2=d(x,y),
\qquad x,y\in V.
\end{equation}
In fact, for $x,y\in V\backslash\{a^\prime,b^\prime\}$ we see from 
\eqref{03def:quadratic embedding Lemma3.3} that
\[
d(x,y)=\|\varphi(x)-\varphi(y)\|^2
=\|\Tilde{\varphi}(x)-\Tilde{\varphi}(y)\|^2.
\]
For $x\in V$ and $y=a^\prime$ we have
\begin{equation}\label{03eqn:in proof (1)}
d(x,a^\prime)=d(x,a)+1
\end{equation}
and 
\begin{equation}\label{03eqn:in proof (2)}
\|\Tilde{\varphi}(x)-\Tilde{\varphi}(a^\prime)\|^2
=\|\varphi(x)-\varphi(a)\|^2+1^2.
\end{equation}
Since $x,a\in V\backslash\{a^\prime,b^\prime\}$ we have
$d(x,a)=\|\varphi(x)-\varphi(a)\|^2$,
combining \eqref{03eqn:in proof (1)}
and \eqref{03eqn:in proof (2)} we come to
\[
d(x,a^\prime)=d(x,a)+1
=\|\varphi(x)-\varphi(a)\|^2+1^2
=\|\Tilde{\varphi}(x)-\Tilde{\varphi}(a^\prime)\|^2.
\]
Similarly, \eqref{03:adding square quadratic embedding}
is shown to be valid for $x\in V$, $y=b^\prime$
and $x=a^\prime$, $y=b^\prime$.
Thus, $\Tilde{\varphi}$ is a quadratic embedding of $G$
and $\mathrm{QEC}(G)\le0$.
Finally, applying \eqref{03eqn:comparison},
we conclude that $\mathrm{QEC}(G)=0$.
\end{proof}

Among the graphs on six vertices the graphs with pendant edges
are easily specified.
They are
G6-15,
G6-18,
G6-26,
G6-34,
G6-35,
G6-46,
G6-52 and G6-75.
Those graphs are of QE class
since all graphs on four vertices are of QE class.

\subsubsection{Extension of a quadratic embedding of a smaller graph}

Let $G$ be a graph on six vertices and $H$ an induced subgraph of $G$
which is isometrically embedded in $G$.
Assume that $H$ is of QE class and given explicitly 
a quadratic embedding in a Euclidean space.
Then we may seek a quadratic embedding of $G$
by extending the given one.
This strategy works well particularly
when $H$ contains a cycle $C_4$ as a (not necessarily induced) subgraph.
Taking G6-79 as an example, we will illustrate the procedure.
\begin{figure}[hbt]
\begin{center}
\includegraphics[width=60pt]{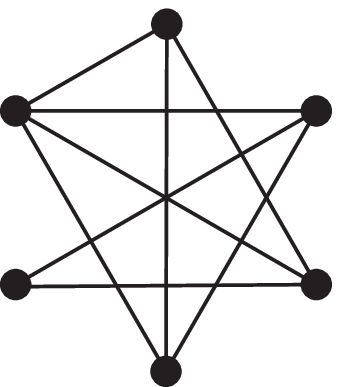}
\qquad \raisebox{20pt}{$\cong$} \qquad
\includegraphics[width=68pt]{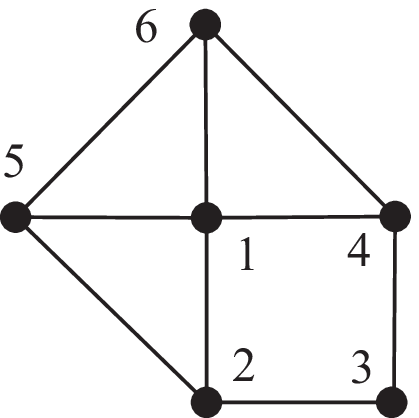}
\end{center}
\caption{G6-79}
\label{explicit quadratic embedding}
\end{figure}

Let $G$ be the graph G6-79 realized on $V=\{1,2,\dots,6\}$
as in Figure \ref{explicit quadratic embedding}.
The induced subgraph $H$ spanned by $\{1,2,3,4\}$ is $C_4$
and isometrically embedded in $G$.
Let $\bm{e}_1,\bm{e}_2,\dots$ be mutually orthogonal
unit vectors in a Euclidean space of sufficiently high dimension.
We note first that a quadratic embedding of $H$ is given explicitly by
\[
\varphi(1)=\bm{0},
\qquad
\varphi(2)=\bm{e}_1,
\qquad
\varphi(3)=\bm{e}_1+\bm{e}_2,
\qquad
\varphi(4)=\bm{e}_2.
\]
For a quadratic embedding of $G$ it is sufficient to find
two vectors in the Euclidean space
corresponding to the vertices 5 and 6,
which are denoted by $\bm{x}$ and $\bm{y}$, respectively.
Then the conditions on $\bm{x}$ and $\bm{y}$ are written down easily.
\begin{gather}
\|\bm{x}-\varphi(1)\|^2=\|\bm{x}-\varphi(2)\|^2=1,
\quad
\|\bm{x}-\varphi(3)\|^2=\|\bm{x}-\varphi(4)\|^2=2, 
\label{03eqn:position x} \\
\|\bm{y}-\varphi(1)\|^2=\|\bm{y}-\varphi(4)\|^2=1,
\quad
\|\bm{y}-\varphi(2)\|^2=\|\bm{y}-\varphi(3)\|^2=2,
\label{03eqn:position y} \\
\|\bm{x}-\bm{y}\|^2=1
\label{03eqn:position xy}
\end{gather}
It follows easily from \eqref{03eqn:position x} that
\[
\langle \bm{x},\bm{e}_1\rangle=\frac12,
\qquad
\langle \bm{x},\bm{e}_2\rangle=0.
\]
Then we set 
\begin{equation}\label{03eqn:x-prime}
\bm{x}
=\langle \bm{x},\bm{e}_1\rangle \bm{e}_1
+\langle \bm{x},\bm{e}_2\rangle \bm{e}_2+\bm{x}^\prime
=\frac{1}{2}\, \bm{e}_1+\bm{x}^\prime,
\end{equation}
where  
\begin{equation}\label{03eqn:x-prime 2}
\langle \bm{x}^\prime,\bm{e}_1\rangle
=\langle \bm{x}^\prime,\bm{e}_2\rangle=0,
\qquad
\frac14+\|\bm{x}^\prime\|^2=1.
\end{equation}
Similarly, we have
\begin{equation}\label{03eqn:y-prime}
\bm{y}
=\frac{1}{2}\, \bm{e}_2+\bm{y}^\prime,
\end{equation}
and
\begin{equation}\label{03eqn:y-prime 2}
\langle \bm{y}^\prime,\bm{e}_1\rangle
=\langle \bm{y}^\prime,\bm{e}_2\rangle=0,
\qquad
\frac14+\|\bm{y}^\prime\|^2=1.
\end{equation}
Inserting \eqref{03eqn:x-prime} and \eqref{03eqn:y-prime}
into \eqref{03eqn:position xy}, we obtain
\begin{equation}\label{02eqn:x^prime and y^prime}
\langle \bm{x}^\prime,\bm{y}^\prime\rangle=\frac12.
\end{equation}
We then see that two vectors $\bm{x}^\prime$ and $\bm{y}^\prime$
satisfying \eqref{03eqn:x-prime 2},
\eqref{03eqn:y-prime 2} and \eqref{02eqn:x^prime and y^prime}
certainly exist since
\[
-1\le 
\frac{\langle \bm{x}^\prime,\bm{y}^\prime\rangle}
     {\|\bm{x}^\prime\|\cdot \|\bm{y}^\prime\|}
=\frac23\le 1.
\]
With such $\bm{x}^\prime$ and $\bm{y}^\prime$ 
we define $\bm{x}$ and $\bm{y}$ by
\eqref{03eqn:x-prime} and \eqref{03eqn:y-prime},
respectively.
Then the quadratic embedding $\varphi$ of $H$ is extended to $G$
by setting $\varphi(5)=\bm{x}$ and $\varphi(6)=\bm{y}$.

\bigskip
The argument in this subsection is applied to the graphs mentioned
in the following table, of which the details are omitted.

\begin{center}
\renewcommand{\arraystretch}{1.2}
\begin{tabular}{|l|l|l|}
\hline
graphs & No. & QE/Non-QE \\
\hline
with pendant edges & G6-15, G6-18, G6-26, G6-34, & QE 
\\
 & G6-35, G6-46, G6-52, G6-75 &
\\ \hline
with an explicit & G6-50, G6-58, G6-63, G6-65, & QE
\\
quadratic embedding & G6-76, G6-79, G6-81, G6-82, & 
\\
& G6-87, G6-90, G6-93, G6-95, & 
\\
& G6-99, G6-100 &
\\
\hline
\end{tabular}
\end{center}

%%%%%%%%%%%%%%%%%%%%%%%%%%%%%%%%%%%%%%%%%%%%%%%
\section{Completion of Classification}

\subsection{Calculating QE Constants}

During the last section we have examined 105 graphs on six vertices
and seven graphs remain for checking,
which are G6-30, G6-32, G6-59, G6-60, G6-66, G6-84 and G6-106,
see Figures \ref{fig:remainders 1}--\ref{fig:remainders 4}.
\begin{figure}[hbt]
\begin{center}
\includegraphics[width=60pt]{6-30.eps}
\quad \raisebox{20pt}{$\cong$} \quad
\includegraphics[width=64pt]{6-30alt.eps}
\qquad
\includegraphics[width=60pt]{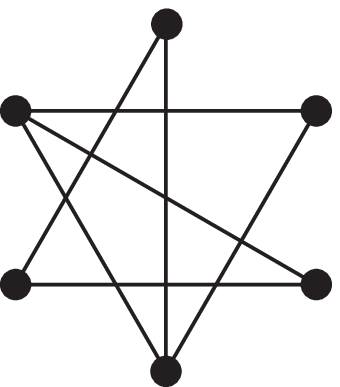}
\quad \raisebox{20pt}{$\cong$} \quad
\includegraphics[width=64pt]{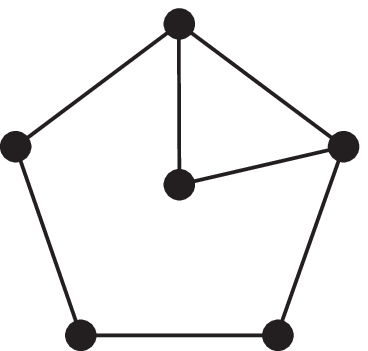}
\end{center}
\caption{$C_6\wedge K_{2,1}$: G6-30 (left) and G6-32 (right)}
\label{fig:remainders 1}
\end{figure}

\begin{figure}[hbt]
\begin{center}
\includegraphics[width=60pt]{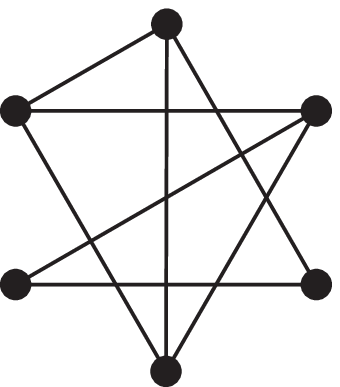}
\quad \raisebox{20pt}{$\cong$} \quad
\includegraphics[width=64pt]{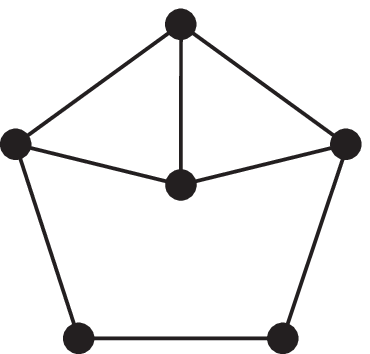}
\qquad
\includegraphics[width=60pt]{6-60.eps}
\quad \raisebox{20pt}{$\cong$} \quad
\includegraphics[width=64pt]{6-60alt.eps}
\end{center}
\caption{$C_6\wedge K_{3,1}$: G6-59 (left) and G6-60 (right)}
\label{fig:remainders 2}
\end{figure}

\begin{figure}[hbt]
\begin{center}
\includegraphics[width=60pt]{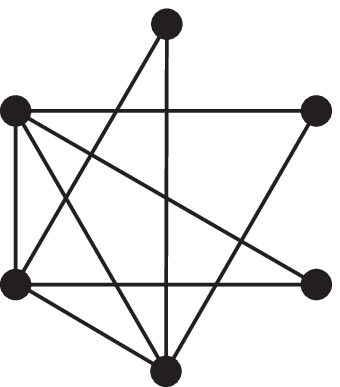}
\quad \raisebox{20pt}{$\cong$} \quad
\includegraphics[width=72pt]{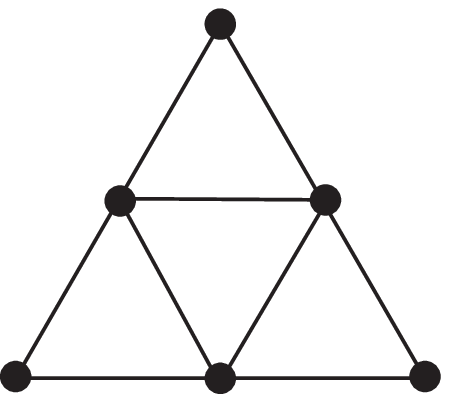}
\qquad
\includegraphics[width=60pt]{6-84.eps}
\quad \raisebox{20pt}{$\cong$} \quad
\includegraphics[width=68pt]{6-84alt.eps}
\end{center}
\caption{G6-66 (left) and G6-84 (right)}
\label{fig:remainders 3}
\end{figure}
\begin{figure}[hbt]
\begin{center}
\includegraphics[width=60pt]{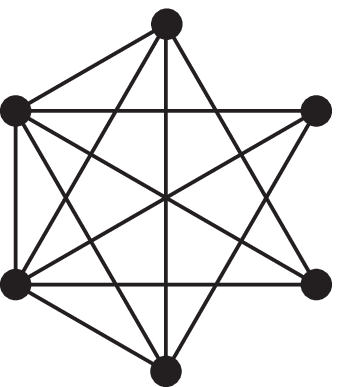}
\end{center}
\caption{G6-106 ($K_6\backslash P_4$)}
\label{fig:remainders 4}
\end{figure}

The QE constants are calculated in a standard manner
(Proposition \ref{02prop:QEC}).
For G6-30, G6-59 and G6-66 the calculation is reduced to 
a quadratic equation and we obtain 
\begin{align*}
\mathrm{QEC}(\text{G6-30})
&=\frac{-4+\sqrt{19}}{3}
\approx 0.1196..., \\
\mathrm{QEC}(\text{G6-59})
&=\frac{-5+\sqrt{19}}{3}
\approx -0.2137..., \\
\mathrm{QEC}(\text{G6-66})
&=\frac{-3+\sqrt{5}}{2}
\approx -0.3819....
\end{align*}
For G6-32, G6-60 and G6-84 the QE constants are solutions to
algebraic equations of higher degrees.
We have
\begin{align*}
\mathrm{QEC}(\text{G6-32})&=\lambda^*_1\approx -0.3121..., \\
\mathrm{QEC}(\text{G6-60})&=\lambda^*_2\approx 0.2034..., \\
\mathrm{QEC}(\text{G6-84})&=\lambda^*_3\approx 0.1313...,
\end{align*}
where $\lambda_1^*,\lambda_2^*$ and $\lambda_3^*$ 
are the largest roots of the equations
\begin{align*}
&3\lambda^3+15\lambda^2+14\lambda+3=0, \\
& 5\lambda^3+26\lambda^2+24\lambda-6=0, \\
&3\lambda^4+14\lambda^3+18\lambda^2+5\lambda-1=0,
\end{align*}
respectively.
Finally, we see that G6-106 is obtained from $K_6$
by deleting three edges which form a path $P_4$,
that is, $\text{G-106}= K_6\backslash P_4$.
Applying the general formula in Appendix B, we obtain
\[
\mathrm{QEC}(\text{G6-106})
=\frac{-3+\sqrt{5}}{2}
\approx -0.3819....
\]

\begin{remark}
\normalfont
Apparently, graphs G6-30, G6-32, G6-59 and G6-60 share 
a common structure, namely, they are obtained by
joining the cycle $C_5$ with the end-vertices of a star $K_{m,1}$.
Although not determined uniquely,
such a graph is denoted by $C_5\wedge K_{m,1}$ for simplicity.
Up to now no useful formula is known
for the QE constants of graphs of this type.
\end{remark}

The results of this subsection is summarized in the following table,
where all graphs $C_5\wedge K_{m,1}$ with $1\le m\le 5$ are
listed for completeness.

\begin{center}
\renewcommand{\arraystretch}{1.2}
\begin{tabular}{|l|l|l|}
\hline
graphs & No. & QE/Non-QE \\
\hline
$C_5\wedge K_{m,1}$  
 & G6-30 ($C_5\wedge K_{2,1}$), G6-60 ($C_5\wedge K_{3,1}$) & Non-QE 
\\ \cline{2-3}
 & G6-16 ($C_5\wedge K_{1,1}$), G6-32 ($C_5\wedge K_{2,1}$), & QE
\\
 & G6-59 ($C_5\wedge K_{3,1}$), G6-79 ($C_5\wedge K_{4,1}$), &
\\
 & G6-92 ($C_5\wedge K_{5,1}$) &
\\ \hline
Remainders &  G6-84 & Non-QE
\\ \cline{2-3}
 &  G6-66, G6-106 ($K_6\backslash P_3$)  & QE
\\ \hline
\end{tabular}
\end{center}

\subsection{Conclusion}

During the last section
all non-primary non-QE graphs on six vertices are found
(Subsection \ref{subsec:Non-Primary Non-QE Graphs})
and no primary non-QE graphs are found.
Hence the three graphs with positive QE constants
found in the last subsection
exhaust the primary non-QE graphs on six vertices,
as is mentioned in Theorem \ref{01thm:primary non-QE graphs}.
The proof of Theorem \ref{01thm:non-primary non-QE graphs}
is already completed 
in Subsection \ref{subsec:Non-Primary Non-QE Graphs}.
Then the assertion of Theorem \ref{01thm:QE graphs} is obvious.

Together with the two primary non-QE graphs on five vertices
(Theorem \ref{02thm:G5-10 and 17})
we have thus determined five primary non-QE graphs 
among the graphs on $n\le 6$ vertices.
It is a natural desire to explore new primary non-QE graphs
on seven or more vertices.
However, this question seems to be very difficult
because of lack of constructive approach.
In fact, our main argument to reach 
a primary non-QE graph on six vertices looks 
like the Eratosthenes' sieve for prime numbers.

In the end of this paper we mention three examples
of primary non-QE graphs on seven vertices.
They are found as a byproduct of the formula 
for the QE constant of the complete multipartite graphs
(Proposition \ref{03prop:complete multipartite graphs}).

\begin{proposition}[\cite{Obata2022}]
\label{01thm:CMG of non-QE}
Among the complete multipartite graphs 
$K_{m_1,m_2,\dots,m_k}$ with $k\ge2$ and $m_1\ge m_2\ge \dotsb \ge m_k\ge1$
there are exactly four primary non-QE graphs,
that are $K_{3,2}\,, K_{5,1,1}\,, K_{4,1,1,1}$ and $K_{3,1,1,1,1}$.
Moreover, any complete multipartite graph of non-QE class
contains at least one of the above four primary ones
as an isometrically embedded subgraph.
\end{proposition}

In other words, $K_{5,1,1}$, $K_{4,1,1,1}$ and $K_{3,1,1,1,1}$
are primary non-QE graphs on seven vertices.
For more details see \cite{Obata2022}.

%%%%%%%%%%%%%%%%%%%%%%%%%%%%%%%%%%%%%%%%%%%%%%%%%%%%%%%%
\setcounter{section}{1}
\section*{Appendix A: Calculating $\mathrm{QEC}(K_n\wedge K_{m,1}$)}
\setcounter{equation}{0}
\setcounter{theorem}{0}
\renewcommand{\thesection}{\Alph{section}}
\renewcommand{\theequation}{\Alph{section}.\arabic{equation}}

For $1\le m\le n$
the graph $K_n\wedge K_{m,1}$ is obtained 
by putting $m$ vertices of the complete graph $K_n$
with the $m$ end-vertices of the star $K_{m,1}$ together.
We will prove the following formula.

\begin{proposition}
For $1\le m\le n$ we have
\[
\mathrm{QEC}(K_n\wedge K_{m,1})
=\frac{-2n+m-1+\sqrt{n(n-m)(m+1)}}{n+1}\,.
\]
\end{proposition}

We adopt a realization of $K_n\wedge K_{m,1}$
on the vertex set $\{0,1,2,\dots,n\}$.
Consider the complete graph $K_n=(V,E)$,
where $V=\{1,2,\dots,n\}$
and $E=\{\{x,y\}\,;\, x,y \in V, x\neq y\}$.
Then $K_n\wedge K_{m,1}=(\Tilde{V},\Tilde{E})$ is given by
\[
\Tilde{V}=V\cup\{0\}=\{0,1,2,\dots,n\},
\qquad
\Tilde{E}=E\cup\{\{0,x\}\,;\, x=1,2,\dots,m\}.
\]
The vertex set of $K_n\wedge K_{m,1}$ being
divided into three parts $\{0\}$,
$\{1,2,\dots,m\}$ and $\{m+1,\dots,n\}$,
the distance matrix of $K_n\wedge K_{m,1}$ admits
a natural block matrix form:
\[
D=
\begin{bmatrix}
0 & \1 & 2\1 \\
\1 & J-I & J \\
2\1 & J & J-I
\end{bmatrix},
\]
where $\1$ is a column or row vector with all the entries are one,
$J$ the matrix with all entries being one,
and $I$ the identity matrix,
and their sizes are understood in the context.

Accordingly, the quadratic form associated to $D$ is written as
\[
\left\langle \begin{bmatrix} f \\ g \\ h \end{bmatrix},
D\begin{bmatrix} f \\ g \\ h \end{bmatrix} \right\rangle,
\qquad
f\in\mathbb{R}, \quad
g=\begin{bmatrix} g_1 \\ \vdots \\ g_m \end{bmatrix}
\in\mathbb{R}^m, \quad
h=\begin{bmatrix} h_1 \\ \vdots \\ h_{n-m} \end{bmatrix}
\in\mathbb{R}^{n-m}.
\]
Upon employing the method of Lagrange's multiplier
we define
\begin{align*}
\varphi(f,g,h,\lambda,\mu)
&=\left\langle \begin{bmatrix} f \\ g \\ h \end{bmatrix},
D\begin{bmatrix} f \\ g \\ h \end{bmatrix} \right\rangle \\[4pt]
&\qquad
-\lambda(f^2+\langle g,g \rangle+\langle h,h \rangle-1)
-\mu(f+\langle \1,g \rangle+\langle \1,h \rangle).
\end{align*}
After simple calculus we see that
any stationary point of $\varphi(f,g,h,\lambda,\mu)$ 
is obtained as a solution to the following system of equations:
\begin{align}
\frac{\partial\varphi}{\partial f}
&=-2\lambda f +2\langle\1,g\rangle
 +4\langle\1,h\rangle-\mu=0,
\label{aeqn:01} \\
\frac{\partial\varphi}{\partial g_i}
&=2f+2\langle\1,g\rangle
 +2\langle\1,h\rangle
 -2(\lambda+1)g_i-\mu=0,
\label{aeqn:02} \\
\frac{\partial\varphi}{\partial h_j}
&=4f+2\langle\1,g\rangle
 +2\langle\1,h\rangle
 -2(\lambda+1)h_j-\mu=0,
\label{aeqn:03}
\end{align}
under the two constraints:
\begin{align}
\frac{\partial\varphi}{\partial \lambda}
&=f^2+\langle g,g\rangle +\langle h,h\rangle-1=0,
\label{aeqn:04} \\
\frac{\partial\varphi}{\partial \mu}
&=f+\langle \1,g\rangle +\langle \1,h\rangle=0.
\label{aeqn:05}
\end{align}
Let $\mathcal{S}$ be the set of all solutions to
\eqref{aeqn:01}--\eqref{aeqn:05}.

Hereafter we assume that $1\le m<n$.
Since $\mathrm{QEC}(K_n\wedge K_{m,1})>-1$
by Proposition \ref{02prop:QEC ge -1},
it is sufficient to
seek out all solutions $(f,g,h,\lambda,\mu)\in\mathcal{S}$ with $\lambda>-1$.
It is obvious from \eqref{aeqn:02} and \eqref{aeqn:03} that
$g_i$ and $h_j$ are constant. We set
\[
g_i=\gamma, \quad 1\le i\le m;
\qquad
h_j=\delta, \quad 1\le j\le n-m.
\]
After simple algebra the equations 
\eqref{aeqn:01}--\eqref{aeqn:05} are reduced to the following
\begin{align}
&-\lambda f +m\gamma+2(n-m)\delta=\frac{\mu}{2}\,,
\label{aeqn:11} \\
&f+(m-1-\lambda)\gamma+(n-m)\delta=\frac{\mu}{2}\,,
\label{aeqn:12} \\
&2f+m\gamma+(n-m-1-\lambda)\delta=\frac{\mu}{2}\,,
\label{aeqn:13} \\
&f^2+m\gamma^2+(n-m)\delta^2=1
\label{aeqn:14} \\
&f+m\gamma+(n-m)\delta=0.
\label{aeqn:15}
\end{align}
From \eqref{aeqn:15} we obtain
\begin{equation}\label{aeqn:25}
f=-m\gamma-(n-m)\delta.
\end{equation}
Then \eqref{aeqn:11}--\eqref{aeqn:13} become
\begin{align}
&m(\lambda+1)\gamma+(n-m)(\lambda+2)\delta=\frac{\mu}{2}\,,
\label{aeqn:21} \\
&-(\lambda+1)\gamma =\frac{\mu}{2}\,,
\label{aeqn:22} \\
&-m\gamma+(-n+m-1-\lambda)\delta=\frac{\mu}{2}\,,
\label{aeqn:23}
\end{align}
respectively.
Set $\mu=0$ in \eqref{aeqn:21}--\eqref{aeqn:23}.
From our condition $1\le m<n$ and $\lambda>-1$ we
see immediately that $\gamma=\delta=0$,
and hence $f=0$ by \eqref{aeqn:25}.
But $f=\gamma=\delta=0$ does not fulfill \eqref{aeqn:14}.
Thus, $\mu\neq0$.
Coming back to \eqref{aeqn:22}, we have
\[
\gamma=\frac{-1}{\lambda+1}\,\frac{\mu}{2}\,.
\]
Then \eqref{aeqn:21} and \eqref{aeqn:23} become
\begin{align}
&(n-m)(\lambda+2)\delta=(m+1)\frac{\mu}{2}\,,
\label{aeqn:31} \\
&(\lambda+n-m+1)\delta=\bigg(\frac{m}{\lambda+1}-1\bigg)\frac{\mu}{2}\,,
\label{aeqn:33}
\end{align}
respectively.
Hence
\[
(\lambda+n-m+1)(m+1)\frac{\mu}{2}
-(n-m)(\lambda+2)\bigg(\frac{m}{\lambda+1}-1\bigg)\frac{\mu}{2}=0.
\]
After simple algebra with $\mu\neq0$ we obtain
\[
(n+1)\lambda^2+2(2n-m+1)\lambda
+m^2-mn+3n-2m+1=0.
\]
The above quadratic equation possesses two real roots given by
\[
\lambda_{\pm}
=\frac{-(2n-m+1)\pm\sqrt{n(n-m)(m+1)}}{n+1}\,.
\]
For $\lambda=\lambda_{\pm}$ we may determine 
$f,\gamma$ and $\delta$ in such a way that
\eqref{aeqn:21}--\eqref{aeqn:23} and \eqref{aeqn:14}
are satisfied.
Thus, we conclude that $\lambda_+$ is the
maximum of $\lambda$ appearing in $\mathcal{S}$.
Namely, 
\begin{equation}\label{aeqn:QEC(Kn wedge Km1)}
\mathrm{QEC}(K_n\wedge K_{m,1})
=\frac{-(2n-m+1)+\sqrt{n(n-m)(m+1)}}{n+1}\,,
\end{equation}
as desired.

For $m=n$ we have $\mathrm{QEC}(K_n\wedge K_{m,1})=\mathrm{QEC}(K_{n+1})=-1$.
On the other hand, the value of the right-hand side of
\eqref{aeqn:QEC(Kn wedge Km1)} for $m=n$ is $-1$.
Consequently, the formula \eqref{aeqn:QEC(Kn wedge Km1)} is valid
for all $1\le m\le n$.

%%%%%%%%%%%%%%%%%%%%%%%%%%%%%%%%%%%%%%%%%%%%%%%%%%%%%%%%%%%%%%
\setcounter{section}{2}
\section*{Appendix B: Calculating $\mathrm{QEC}(K_n\backslash P_4)$}
\setcounter{equation}{0}
\setcounter{theorem}{0}
\renewcommand{\thesection}{\Alph{section}}
\renewcommand{\theequation}{\Alph{section}.\arabic{equation}}

For $n\ge5$ 
let $K_n\backslash P_4$ be the graph obtained
by deleting three edges of $K_n$
chosen in such a way that the subgraph generated by them
is isomorphic to the path $P_4$.

\begin{proposition}
We have
\[
\mathrm{QEC}(K_n\backslash P_4)
=\frac{-3+\sqrt{5}}{2}\,,
\qquad n=5,6,
\]
and 
\[
\mathrm{QEC}(K_n\backslash P_4)
=\frac{-(n+6)+\sqrt{(n+6)^2+4n(n-10)}}{2n}\,,
\qquad n\ge 7.
\]
\end{proposition}

Upon explicit calculation we consider the complete graph
$K_n=(V,\Tilde{E})$ with vertex set $V=\{1,2,\dots,n\}$
and realize $K_n\backslash P_4$ as $G=(V,E)$, where 
\[
E=\Tilde{E}\backslash\{\{1,2\},\{2,3\},\{3,4\}\}.
\]
The distance matrix of $G$ is written in a block matrix form:
\begin{equation}
D=\begin{bmatrix}
A & J \\
J & J-I
\end{bmatrix},
\qquad
A=\begin{bmatrix}
0 & 2 & 1 & 1 \\
2 & 0 & 2 & 1 \\
1 & 2 & 0 & 2 \\
1 & 1 & 2 & 0
\end{bmatrix}
\end{equation}
Then the quadratic form associated to $D$ is written as
\[
\left\langle \begin{bmatrix} f \\ g  \end{bmatrix},
D\begin{bmatrix} f \\ g \end{bmatrix} \right\rangle
=\langle f, Af\rangle
 +2\langle \1, f\rangle \langle \1, g\rangle
 +\langle \1, g\rangle^2
 -\langle g, g\rangle,
\]
where 
\[
f=\begin{bmatrix} f_1 \\ \vdots \\ f_4 \end{bmatrix} \in\mathbb{R}^4,
\qquad
g=\begin{bmatrix} g_1 \\ \vdots \\ g_{n-4} \end{bmatrix} 
\in\mathbb{R}^{n-4}.
\]
For $\lambda,\mu\in \mathbb{R}$ we set
\begin{align*}
\varphi(f,g,\lambda,\mu)
&=\langle f, Af\rangle
 +2\langle \1, f\rangle \langle \1, g\rangle
 +\langle \1, g\rangle^2
 -\langle g, g\rangle \\
&\qquad
-\lambda(\langle f,f \rangle+\langle g,g \rangle-1)
-\mu(\langle \1,f \rangle+\langle \1,g \rangle).
\end{align*}
A stationary point of $\varphi(f,g,\lambda,\mu)$ is
characterized as a solution to the equations:
\begin{equation}\label{2aeqn:original equations}
\frac{\partial\varphi}{\partial f_i}
=\frac{\partial\varphi}{\partial g_j}=0,
\qquad 1\le  i\le 4,
\quad 1\le j\le n-4,
\end{equation}
under the two constraints:
\begin{align}
\langle f,f \rangle+\langle g,g \rangle=1, 
\label{2eqn:constraint1} \\
\langle \1,f \rangle+\langle \1,g \rangle=0.
\label{2eqn:constraint2} 
\end{align}
After direct computation of derivatives and
application of \eqref{2eqn:constraint2}
the equations \eqref{2aeqn:original equations} become
\begin{align}
&-(\lambda+1)f_1+f_2=\frac{\mu}{2},
\label{2aeqn:eqn1} \\
&f_1-(\lambda+1)f_2+f_3=\frac{\mu}{2},
\label{2aeqn:eqn2} \\
&f_2-(\lambda+1)f_3+f_4=\frac{\mu}{2},
\label{2aeqn:eqn3} \\
&f_3-(\lambda+1)f_4=\frac{\mu}{2},
\label{2aeqn:eqn4} \\
&-(\lambda+1)g_j=\frac{\mu}{2}\,,
\qquad 1\le j\le n-4.
\label{2aeqn:eqn5}
\end{align}
Let $\mathcal{S}$ be the set of all solutions
$(f,g,\lambda,\mu)$ to the equations
\eqref{2aeqn:eqn1}--\eqref{2aeqn:eqn5}
under the two constraints
\eqref{2eqn:constraint1} and \eqref{2eqn:constraint2}.
By Propositions \ref{02prop:QEC} and \ref{02prop:QEC ge -1}
it is sufficient to seek a solution $(f,g,\lambda,\mu)$ with $\lambda>-1$.

We first see from \eqref{2aeqn:eqn5} that
$g_j=\gamma$ is constant independently of $j$ and
\begin{equation}\label{aeqn:g_j}
g_j=\gamma=\frac{-1}{\lambda+1}\,\frac{\mu}{2}\,.
\end{equation}
With the help of \eqref{2aeqn:eqn1} and \eqref{2aeqn:eqn4}
the variables $f_2$ and $f_3$ being eliminated,
\eqref{2aeqn:eqn2} and \eqref{2aeqn:eqn3} become
\begin{align}
&(-\lambda^2-2\lambda)f_1+(\lambda+1)f_4=(\lambda+1)\,\frac{\mu}{2},
\label{2aeqn:eqn6} \\
&(\lambda+1)f_1+(-\lambda^2-2\lambda)f_4=(\lambda+1)\,\frac{\mu}{2},
\label{2aeqn:eqn7}
\end{align}
respectively.
The determinant of the coefficients of the left-hand side
is given by
\[
\Delta(\lambda)=(-\lambda^2-2\lambda)^2-(\lambda+1)^2
=(\lambda^2+3\lambda+1)(\lambda^2+\lambda-1).
\]

(Case 1) $\Delta(\lambda)\neq0$.
The solution to \eqref{2aeqn:eqn6} and \eqref{2aeqn:eqn7} is 
unique and we obtain
\begin{equation}\label{aeqn:solutions f_i}
f_1=f_4=\frac{-(\lambda+1)}{\lambda^2+\lambda-1}\,\frac{\mu}{2}\,,
\qquad
f_2=f_3=\frac{-(\lambda+2)}{\lambda^2+\lambda-1}\,\frac{\mu}{2}\,.
\end{equation}
Then, using \eqref{aeqn:g_j} and \eqref{aeqn:solutions f_i},
the constraint \eqref{2eqn:constraint2} becomes
\begin{equation}\label{2aeqn:lambda and mu}
\bigg(\frac{4\lambda+6}{\lambda^2+\lambda-1}
+\frac{n-4}{\lambda+1}\bigg)\,\frac{\mu}{2}=0.
\end{equation}
If $\mu=0$, 
it follows from \eqref{aeqn:g_j} and \eqref{aeqn:solutions f_i}
that $f_i=g_j=0$,
which does not fulfill the constraint \eqref{2eqn:constraint1}.
Hence $\mu\neq0$ and from \eqref{2aeqn:lambda and mu} we obtain
\[
n\lambda^2+(n+6)\lambda-(n-10)=0.
\]
The roots are real and given by
\begin{equation}\label{2aeqn:lambda+-}
\lambda_{\pm}
=\frac{-(n+6)\pm \sqrt{(n+6)^2+4n(n-10)}}{2n}\,.
\end{equation}
For $\lambda=\lambda_\pm$ we have $\Delta(\lambda_\pm)\neq0$
and we see that
$f_i$ and $g_j$ are non-zero multiples of $\mu$.
Hence we may determine $\mu$ in such a way that 
\eqref{2eqn:constraint1} is satisfied.
Consequently, $\lambda_{\pm}$ appear in $\mathcal{S}$.

(Case 2) $\lambda^2+3\lambda+1=0$.
By $\lambda^2+2\lambda=-\lambda-1$ and $\lambda\neq -1$
it is easy to see that
\eqref{2aeqn:eqn6} and \eqref{2aeqn:eqn7} are reduced to 
\begin{equation}\label{2aeqn:f_1+f_4}
f_1+f_4=\frac{\mu}{2}\,.
\end{equation}
From \eqref{2aeqn:eqn1} and \eqref{2aeqn:eqn4} we have
\begin{equation}\label{2aeqn:f_2 and f_3}
f_2=(\lambda+1)f_1+\frac{\mu}{2},
\qquad
f_3=(\lambda+1)f_4+\frac{\mu}{2}
\end{equation}
so that
\begin{equation}\label{2aeqn:f_2+f_3}
f_2+f_3=(\lambda+3)\,\frac{\mu}{2}\,.
\end{equation}
From \eqref{2aeqn:f_1+f_4}, \eqref{2aeqn:f_2+f_3}
and \eqref{aeqn:g_j} we see that
\[
\langle \1,f\rangle=(\lambda+1)\,\frac{\mu}{2}\,,
\qquad
\langle \1,g\rangle=\frac{-(n-4)}{\lambda+1}\,\frac{\mu}{2}\,.
\]
Then the constraint \eqref{2eqn:constraint2} becomes
\[
\frac{1}{\lambda+1}
\big(\lambda^2+5\lambda-(n-8)\big)\,\frac{\mu}{2}=0.
\]
Since there is no common root of
$\lambda^2+3\lambda+1=0$ and
$\lambda^2+5\lambda-(n-8)=0$,  we obtain $\mu=0$.
Then by \eqref{2aeqn:f_1+f_4} and \eqref{2aeqn:f_2 and f_3}
we obtain
\[
f_2=(\lambda+1)f_1\,,
\qquad
f_3=(\lambda+1)f_4=-(\lambda+1)f_1\,,
\qquad 
f_4=-f_1\,.
\]
The second constraint \eqref{2eqn:constraint1} becomes
\[
f_1^2+(\lambda+1)^2f_1^2+(\lambda+1)^2f_1+f_1^2=1
\]
and $f_1$ is determined.
Thus, the root of $\lambda^2+3\lambda+1=0$,
that is,
\[
\lambda=\frac{-3\pm\sqrt{5}}{2}
\]
appears in $\mathcal{S}$.

(Case 3) $\lambda^2+\lambda-1=0$.
In a similar manner as in Case 2, we obtain
\[
\mu=0,
\qquad
g_j=\gamma=0,
\qquad
f_1=f_2=(\lambda+1)f_3,
\qquad
f_4=f_3.
\]
For the constraint \eqref{2eqn:constraint2} 
we have necessarily $f_i=g_j=0$, which however
do not fulfill \eqref{2eqn:constraint1}.
Namely, any root of $\lambda^2+\lambda-1=0$ does not appear
in $\mathcal{S}$.

From (Case 1)--(Case 3) we conclude that
\[
\mathrm{QEC}(G)
=\max\left\{\lambda_+, \frac{-3+\sqrt{5}}{2} \right\},
\]
where $\lambda_+$ is given as in \eqref{2aeqn:lambda+-}.
Finally, comparison of two numbers in the 
right-hand side is easy by simple algebra 
and we obtain the desired formula.

%%%%%%%%%%%%%%%%%%%%%%%%%%%%%%%%%%%%%%%%%%%%%%%%%%%%%%%%%%%%%%
\setcounter{section}{3}
\section*{Appendix C: List of All Graphs on Six Vertices}
\setcounter{equation}{0}
\setcounter{theorem}{0}
\renewcommand{\thesection}{\Alph{section}}
\renewcommand{\theequation}{\Alph{section}.\arabic{equation}}

For the self-contained reference we reproduce 
the list of all connected graphs on six vertices
following McKay \cite{McKay}.
Another list by Cvetkovi\'c--Petri\'c \cite{Cvetkovic-Petric1984}
is also useful, where the eigenvalues of the adjacency matrices
and related quantities are mentioned.

\begin{center}
\includegraphics[width=400pt]{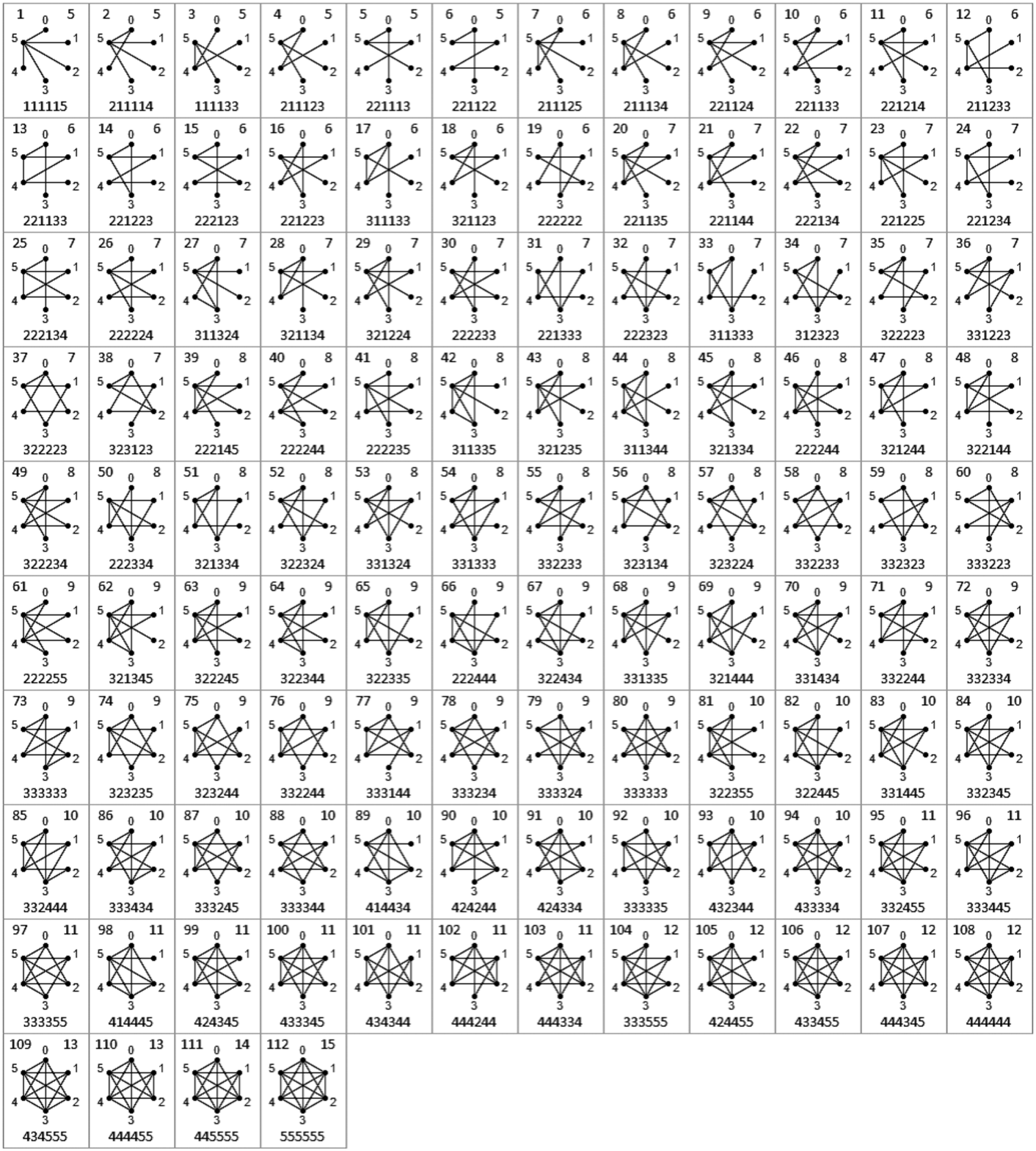}
\end{center}

%%%%%%%%%%%%%%%%%%%%%%%%%%%%%%%%%%%%%%%%%%%

\end{document}